\documentclass[12pt]{amsart}
\usepackage{mathtools}
\usepackage{fullpage}
\usepackage[alphabetic,nobysame, initials]{amsrefs}
\usepackage{enumerate}
\usepackage{hyperref}
\usepackage{color}
\usepackage[pdftex]{graphicx}

\newtheorem{Theorem}{Theorem}[section]
\newtheorem{Lemma}[Theorem]{Lemma}
\newtheorem{Proposition}[Theorem]{Proposition}

\newtheorem{Conjecture}{Conjecture}
\newtheorem{Problem}{Problem}

\renewcommand{\cir}{{\rm cr}}

\newcommand{\vect}[1]{{\bf #1}}
\newcommand{\di}{{\mathrm d}}
\newcommand{\Ee}{{\mathbb E}}
\newcommand{\Ed}{\Ee^d}
\newcommand{\len}{{\rm length}}

\newcommand{\iprod}[2]{\left<#1,#2\right>}

\newcommand{\noshow}[1]{}
\newcommand{\B}{{\mathbf B}}
\newcommand{\Sedm}{{\mathbb S}^{d-1}}
\newcommand{\st}{\; : \; }
\newcommand{\Ze}{{\mathbb Z}}
\newcommand{\Edn}{\Ee^{d\times N}}
\newcommand{\conv}{{\rm conv}}
\newcommand{\ivol}[2][k]{{\rm V}_{#1}\left(#2\right)}
\newcommand{\kind}{k\in[d]}

\newcommand{\ivolbo}{\ivol{\B[o]}}

\title{The Kneser--Poulsen conjecture for special contractions}

\author{K\'{a}roly Bezdek and M\'arton Nasz\'odi}

\address{
K\'{a}roly Bezdek,
Department of Mathematics and Statistics, University of Calgary, Canada,
and
Department of Mathematics, University of Pannonia, Veszpr\'em, Hungary.
}
\email{bezdek@math.ucalgary.ca}

\address{
M\'arton Nasz\'odi,
Dept. of Geometry,
Lorand E\"otv\"os University, Budapest.
}
\email{marton.naszodi@math.elte.hu}

\keywords{Kneser--Poulsen conjecture, Alexander's contraction, ball-polyhedra, 
volume of intersections of balls, volume of unions of balls, Blaschke--Santalo 
inequality}
\subjclass[2010]{52A20,52A22}

\begin{document}
\begin{abstract}
The Kneser--Poulsen Conjecture states that if the centers of a family of 
$N$ unit balls in ${\mathbb E}^d$ is contracted, then the volume of the union 
(resp., intersection) does not increase (resp., decrease). We consider two 
types of special contractions. 

First, a \emph{uniform contraction} is a contraction where all the 
pairwise distances in the first set of centers are larger than all the pairwise 
distances in the second set of centers. We obtain that 
a uniform contraction of the centers does not decrease the volume of the 
intersection of the balls, provided that $N\geq(1+\sqrt{2})^d$. Our result 
extends to intrinsic volumes. We prove a similar result concerning the volume 
of the union.

Second, a \emph{strong contraction} is a contraction in each coordinate. We 
show that the conjecture holds for strong contractions. In fact, the result 
extends to arbitrary unconditional bodies in the place of balls.
\end{abstract}
\maketitle

\section{Introduction}

We denote the Euclidean norm of a vector $p$ in the $d$-dimensional Euclidean 
space $\Ed$ by $|p|:=\sqrt{\iprod{p}{p}}$, where $\iprod{\cdot}{\cdot}$ is the 
standard inner product. 
For a positive integer $N$, we use $[N]=\{1,2,\ldots,N\}$.
Let $A\subset\Ed$ be a set, and 
$\kind$. We denote the $k-$th intrinsic volume
of $A$ by $\ivol{A}$; in particular, $\ivol[d]{A}$ is the $d-$dimensional 
volume.
The closed Euclidean ball of radius $\rho$ centered at $p\in\Ed$ is denoted 
by $\B[p,\rho]:=\{q\in\Ed\st |p-q|\leq\rho\}$, its volume is 
$\rho^d\kappa_d$, where $\kappa_d:=\ivol[d]{\B[o,1]}$.
For a set $X\subset\Ed$, the intersection of balls of radius 
$\rho$ around the points in $X$ is $\B[X,\rho]:=\cap_{x\in X} 
\B[x,\rho]$; when $\rho$ is omitted, then $\rho=1$. The \emph{circumradius} 
$\cir(X)$ of $X$ is the radius of the 
smallest ball containing $X$. Clearly, $\B[X,\rho]$ is empty, if, and only 
if, $\cir(X)>\rho$. We denote the unit sphere centered at the origin 
$o\in\Ed$ by $\Sedm:=\{u\in\Ed\st |u|=1\}$.

It is convenient to denote the (finite) point 
configuration consisting of $N$ points $p_1, p_2, \ldots, p_N$ 
in $\Ed$ by $\vect{p}=(p_1, \ldots, p_N)$, also considered as a point in 
$\Edn$. Now, if $\vect{p}=(p_1, \ldots, p_N)$ and $\vect{q}=(q_1, \ldots, q_N)$ 
are two configurations of $N$ points in 
$\Ed$ such that for all $1\le i<j\le N$ the inequality $|q_i- 
q_j |\le |p_i- p_j |$ holds, then we say that $\vect{q}$ is 
a \emph{contraction} of $\vect{p}$. If $\vect{q}$ is a contraction of 
$\vect{p}$, then there may or may not be a continuous motion 
$\vect{p}(t)=(p_1(t), \ldots, p_N(t))$, with  $p_i(t)\in \Ed$ for all 
$0\le t\le 1$ and $1\le i\le N$ such that $\vect{p}(0)=\vect{p}$ and 
$\vect{p}(1)=\vect{q}$, and $|p_i(t)- p_j(t)|$ is monotone decreasing for all 
$1\le i<j\le N$. When there is such a motion, we say that $\vect{q}$ is a 
\emph{continuous contraction} of $\vect{p}$. 

In 1954 Poulsen \cite{Po} and in 1955 Kneser \cite{Kn} independently 
conjectured the following.

\begin{Conjecture}\label{Kneser-Poulsen}
If $\vect{q}=(q_1,\ldots,q_N)$ is a contraction of $\vect{p}=(p_1,\ldots, p_N)$ 
in $\Ed$, then
\begin{equation*}
\ivol[d]{\bigcup_{i=1}^{N}\mathbf{B}[p_i]}\ge  
\ivol[d]{\bigcup_{i=1}^{N}\mathbf{B}[q_i]}.
\end{equation*}
\end{Conjecture}

A similar conjecture was proposed by Gromov \cite{Gr87} and also by Klee and 
Wagon \cite{KW}.

\begin{Conjecture}
\label{Klee-Wagon}
If $\vect{q}=(q_1,\ldots,q_N)$ is a contraction of $\vect{p}=(p_1,\ldots, p_N)$ 
in $\Ed$, then
\begin{equation*}
\ivol[d]{\bigcap_{i=1}^{N}\mathbf{B}[p_i]}\le  
\ivol[d]{\bigcap_{i=1}^{N}\mathbf{B}[q_i]}.
\end{equation*}
\end{Conjecture}

In fact, both conjectures have been stated for the case of non-congruent balls.

Conjecture~1 is false in dimension $d=2$ when the volume $V_2$ is replaced by 
$V_1$: Habicht and Kneser gave an example (see details in \cite{BeCo}) where 
the centers of a finite family of unit disks on the plane is contracted, and 
the union of the second family is of larger perimeter than the union of the 
first.
On the other hand, Alexander \cite{Al85} conjectured that under any contraction 
of the center points of a finite family of unit disks in the plane, the 
perimeter of the intersection does not decrease. We pose the following more 
general problem.
\begin{Problem}\label{prob:intrinsicKP}
Is it true that whenever $\vect{q}=(q_1,\ldots,q_N)$ is a contraction of 
$\vect{p}=(p_1,\ldots, p_N)$ 
in $\Ed$, then
\begin{equation*}
\ivol{\bigcap_{i=1}^{N}\mathbf{B}[p_i]}\le  
\ivol{\bigcap_{i=1}^{N}\mathbf{B}[q_i]}
\end{equation*}
holds for any $\kind$?
\end{Problem}

For a recent comprehensive overview on the status of 
Conjectures~\ref{Kneser-Poulsen} and~\ref{Klee-Wagon}, which are often called 
the Kneser--Poulsen conjecture in short, 
we refer the interested reader to \cite{B13}. Here, we mention the 
following two results only, which briefly summarize the status of the 
Kneser--Poulsen conjecture. In \cite{Cs1}, Csik\'os proved 
Conjectures~\ref{Kneser-Poulsen} and~\ref{Klee-Wagon} for {\it continuous} 
contractions in all dimensions. 
On the other hand, in \cite{BeCo} the first named author jointly with Connelly 
proved Conjectures~\ref{Kneser-Poulsen} and~\ref{Klee-Wagon} for {\it all} 
contractions in the Euclidean plane. However, the 
Kneser--Poulsen conjecture remains open in dimensions three and higher. 

\subsection{The Kneser-Poulsen conjecture for uniform contractions}

We will investigate Conjectures~\ref{Kneser-Poulsen} and~\ref{Klee-Wagon} 
and Problem~\ref{prob:intrinsicKP} for special contractions of the 
following type. We say that $\vect{q}\in\Edn$ is a \emph{uniform contraction of 
$\vect{p}\in\Edn$ with separating value $\lambda>0$}, if 
\begin{equation}\label{eq:unifcontractiondef}
|q_i-q_j|\leq\lambda\leq|p_i-p_j| \mbox{ for all } i,j\in [N],i\neq j.  
\tag{UC}
\end{equation}

Our first main result is the following.

\begin{Theorem}\label{thm:unifKPmain}
 Let $d, N\in\Ze^+, \kind$, and
let $\vect{q}\in\Edn$ be a uniform contraction of $\vect{p}\in\Edn$ with some 
separating value $\lambda\in(0,2]$. If $ N\geq \left(1+\sqrt{2}\right)^d$ then 
\begin{equation}\label{eq:unifKPmain}
\ivol{\bigcap_{i=1}^{N}\B[p_i]}\le 
\ivol{\bigcap_{i=1}^{N}\B[q_i]}. 
\end{equation}
\end{Theorem}

The strength of this result is its independence of the separating value 
$\lambda$. 

The idea of considering uniform contractions came from a conversation with 
Peter Pivovarov, who pointed out that such conditions arise naturally when 
sampling the point-sets $\vect{p}$ and $\vect{q}$ randomly. If one could find 
distributions for $\vect{p}$ and $\vect{q}$ that satisfy the reversal of 
\eqref{eq:unifKPmain} for $k=d$, while simultaneously satisfying 
\eqref{eq:unifcontractiondef} (with some positive probability), it would lead 
to a counter-example to Conjecture~\ref{Klee-Wagon}. Related problems, 
isoperimetric inequalities for the volume of random ball polyhedra, were 
studied in \cite{PaPi16}.

Our second main result is the proof of Conjecture~\ref{Kneser-Poulsen} under 
conditions analogous to those in Theorem~\ref{thm:unifKPmain}. 

\begin{Theorem}\label{thm:unifKPunionmain}
Let $d, N\in\Ze^+$, and
let $\vect{q}\in\Edn$ be a uniform contraction of $\vect{p}\in\Edn$ with some 
separating value $\lambda\in(0,2]$. If $ N\geq \left(1+2d^3\right)^d$ then 
	\begin{equation}\label{eq:unifKPunionmain}
	\ivol[d]{\bigcup_{i=1}^{N}\B[p_i]}\geq 
	\ivol[d]{\bigcup_{i=1}^{N}\B[q_i]}.
	\end{equation}
\end{Theorem}

Again, the strength of this result is its independence of the separating value 
$\lambda$. Most likely, a more careful computation than the one presented here 
will give a condition $N\geq d^{cd}$ with a universal constant $c$ below 3. It 
would be very interesting to see an exponential condition, that is, one of the 
form $N\geq e^{cd}$.

We note that if $d, N\in\Ze^+$, and $\vect{q}\in\Edn$ is a uniform contraction 
of $\vect{p}\in\Edn$ with some separating value $\lambda\in[2,+\infty)$, then 
\eqref{eq:unifKPunionmain} holds trivially.

\subsection{The Kneser-Poulsen conjecture for strong contractions}

Let us refer to the coordinates of the point $x\in \Ed$ by 
writing $x=(x^{(1)},\ldots , x^{(d)})$.  Now, if 
$\vect{p}=(p_1,\ldots, p_N)$ and $\vect{q}=(q_1,\ldots, q_N)$ are two 
configurations of $N$ points in $\Ed$ such that for all $1\le k\le d$ and $1\le 
i<j\le N$ the 
inequality $|q_i^{(k)}- q_j^{(k)} |\le |p_i^{(k)}-p_j^{(k)} |$ holds, then we 
say that $\vect{q}$ is a \emph{strong contraction} of $\vect{p}$. Clearly, if 
$\vect{q}$ is a strong contraction of $\vect{p}$, then $\vect{q}$ is a 
contraction of $\vect{p}$ as well. 

We describe a non-trivial example of strong contractions.
Let $H:=\{x=(x^{(1)},\ldots , x^{(d)})\in \Ed\st x^{(i)}=h\}$ be a hyperplane 
of $\Ed$ orthogonal to the $i$th 
coordinate axis in $\Ed$. Moreover, let $R_H: 
\Ed\to\Ed$ denote the reflection about $H$ in 
$\Ed$. Furthermore, let $H^{+}:=\{x=(x^{(1)},\ldots 
, x^{(d)})\in \Ed\st x^{(i)}>h\}$ and $H^{-}:=\{x=(x^{(1)},\ldots , x^{(d)})\in 
\Ed\st x^{(i)}<h\}$ be the two open 
halfspaces bounded by $H$ in $\Ed$. Now, let us introduce the \emph{one-sided 
reflection about $H^+$} as the mapping $C_{H^+}: 
\Ed\to\Ed$ defined as follows: If $x\in H\cup 
H_-$, then let $C_{H^+}(x):=x$, and if $x\in H^+$, 
then let $C_{H^+}(x):=R_H(x)$. 
Clearly, for any point configuration $\vect{p}=(p_1,\ldots, p_N)$ of 
$N$ points in $\Ed$ the point 
configuration $\vect{q}:=(C_{H^+}(p_1),\ldots 
, C_{H^+}(p_N))$ is a strong contraction of $\vect{p}$ in 
$\Ed$. 

Clearly, if $H_1, \ldots , H_k$ is 
a sequence of hyperplanes in $\Ed$ each being orthogonal to some of 
the $d$ coordinate axis of $\Ed$, then the composite mapping $ 
C_{H_k^+}\circ \ldots \circ C_{H_2^+}\circ C_{H_1^+}$ is a strong contraction 
of $\Ed$. 

We note that the converse of this statement does not hold. Indeed, 
$\vect{q}=(-100, -1, 0, 99)$ is a strong 
contraction of the point configuration $\vect{p}= (-100,-1, 1, 100)$ in 
$\Ee^{1}$, which cannot be obtained in the 
form $C_{H_k^+}\circ \ldots \circ C_{H_2^+}\circ C_{H_1^+}$ in 
$\Ee^{1}$.   

The question whether Conjectures~\ref{Kneser-Poulsen} and~\ref{Klee-Wagon} hold 
for strong contractions, is a natural one. In what follows we give an 
affirmative answer to that question. We do a bit more. Recall that 
a {\it convex body} in $\Ed$ is a compact convex set with non-empty interior. 
A convex body $K$ is called an \emph{unconditional} 
(or, \emph{$1$-unconditional}) convex body if for any $x=(x^{(1)}, \ldots 
,x^{(d)})\in K$ also $(\pm x^{(1)}, \ldots ,\pm x^{(d)})\in K$ holds. Clearly, 
if $K$ is an unconditional convex body in $\Ed$, then $K$ is symmetric about 
the origin $o$ of  $\Ed$. Our third main result is a generalization of the 
Kneser--Poulsen-type results published in \cite{Bo} and \cite{Re}.

\begin{Theorem}\label{thm:strong}
Let $K_1, \ldots , K_N$ be (not necessarily 
distinct) unconditional convex bodies in $\Ed$, $d\ge 2$.
If $\vect{q}=(q_1, \ldots, q_N)$ is a strong contraction 
of $\vect{p}=(p_1, \ldots, p_N)$ in $\Ed$, 
then
\begin{equation}\label{eq:strong1}
\ivol[d]{\bigcup_{i=1}^{N}(p_i+K_i)}\ge
\ivol[d]{\bigcup_{i=1}^{N}(q_i+K_i)},
\end{equation}
and
\begin{equation}\label{eq:strong2}
\ivol[d]{\bigcap_{i=1}^{N}(p_i+K_i)}\le
\ivol[d]{\bigcap_{i=1}^{N}(q_i+K_i)}.
\end{equation}
\end{Theorem}

We note that the assumption that the bodies are unconditional cannot be dropped 
in Theorem~\ref{thm:strong}. Indeed, Figure~\ref{fig:triangles} shows
two families of translates of a triangle. Both configurations of the three 
translation vectors are a strong contraction of the other configuration. The 
intersection of the first family is a small triangle, while the intersection of 
the second is a point. Additionally, the union of the first family is of larger 
area (resp., perimeter) than the union of the second.

Also note that in Theorem~\ref{thm:strong} we cannot replace volume by surface 
area. Indeed, Figure~\ref{fig:uncondCounterex} shows
two families of translates of unconditional planar convex bodies. The second 
family is a contraction of the first, while the union of the second family is 
of larger perimeter than the union of the first.

\noshow{
\begin{figure}[b]\begin{center}
  \includegraphics[width=.25\textwidth]{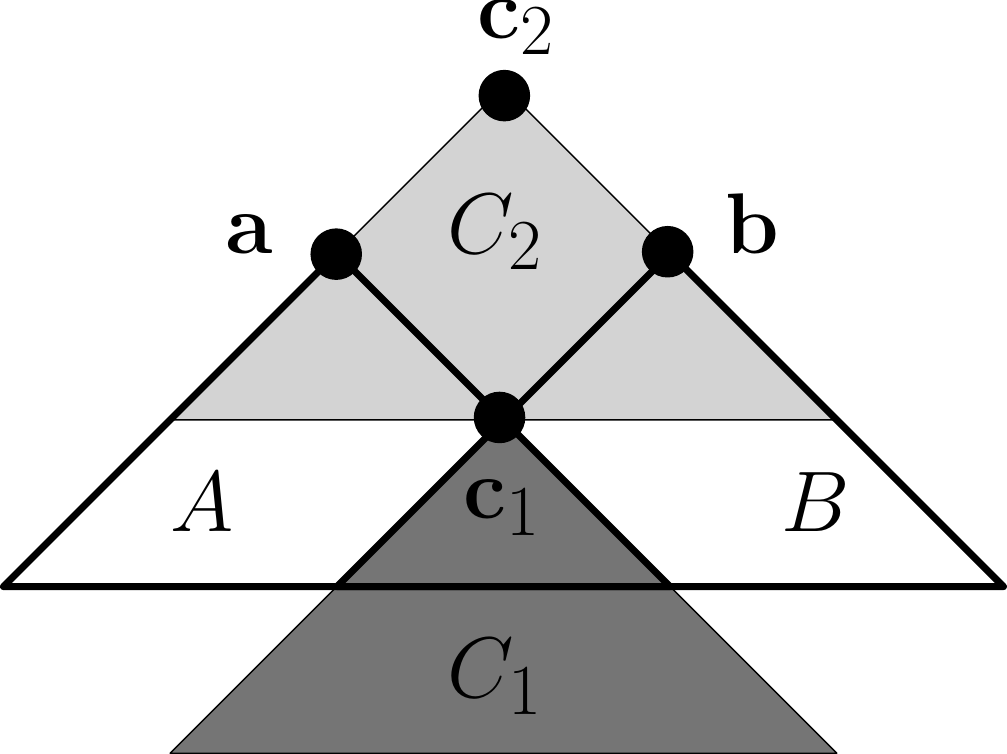}
  \caption{First family of translates of a triangle: $A,B,C_1$; second family: 
$A,B,C_2$, where, for the translation vectors, we have $b=-a$, 
and 
$c_2=-c_1$. Both configurations of the three translation vectors 
are a strong contraction of the other configuration.}\label{fig:triangles}
\end{center}\end{figure}

\begin{figure}[b]\begin{center}
  \includegraphics[width=.25\textwidth]{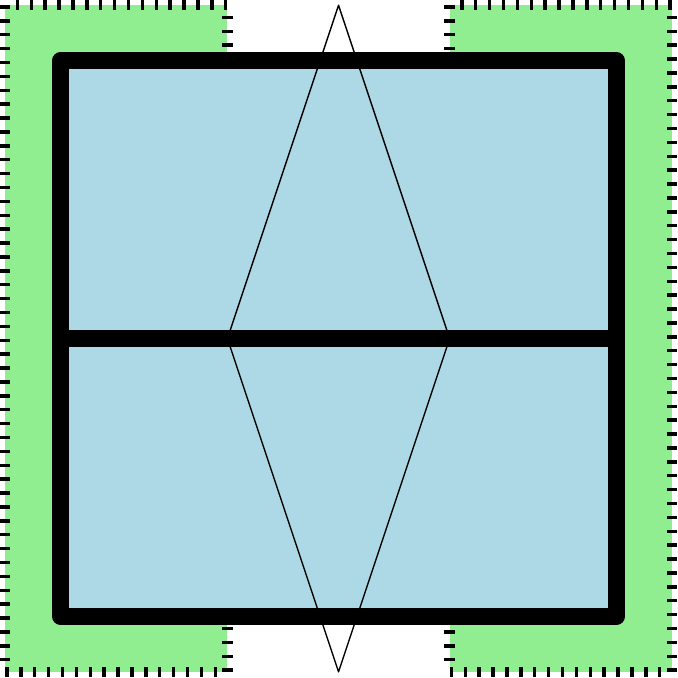}
\par\caption{First family of unconditional sets: The two vertical rectangles, 
the 
two horizontal rectangles and the diamond in the middle; second family: 
The two vertical rectangles, the upper horizontal rectangle taken twice (once 
as itself, and once as a translate of the lower horizontal rectangle) and the 
diamond in the middle.}\label{fig:uncondCounterex}
\end{center}\end{figure}
}

\begin{figure}[bth]
\begin{minipage}[t]{0.55\linewidth}
\begin{center}
\includegraphics[width=0.5\textwidth]{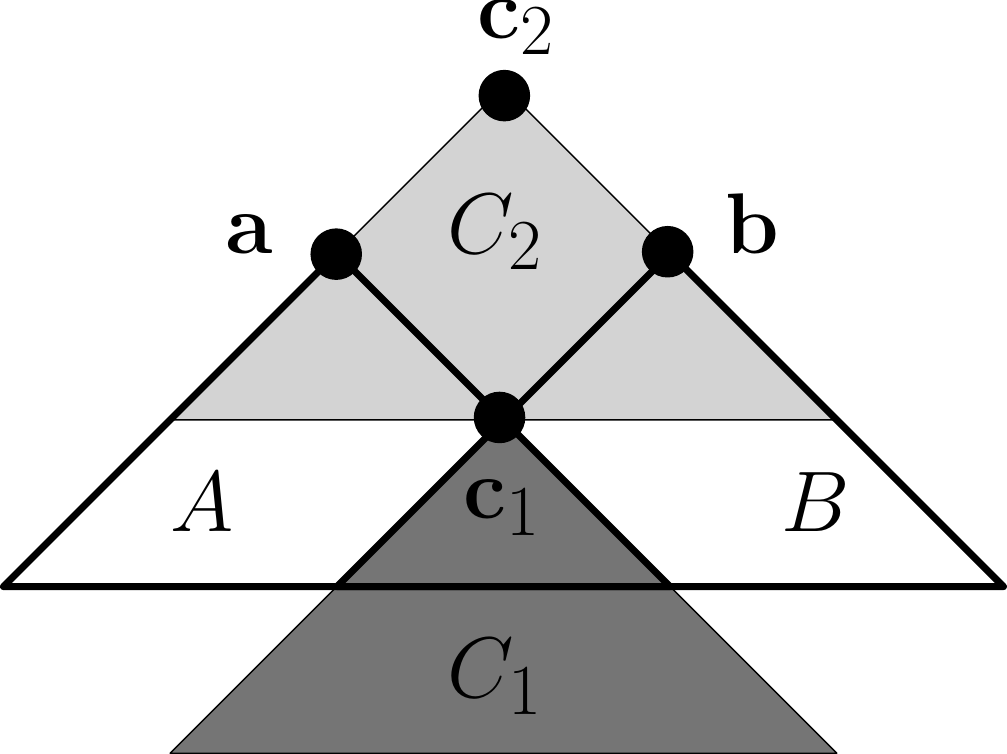}
\par\caption{First family of translates of a triangle: $A,B,C_1$; second 
family: $A,B,C_2$, where, for the translation vectors, we have $b=-a$, and 
$c_2=-c_1$. Both configurations of the three translation vectors are a strong 
contraction of the other configuration.}\label{fig:triangles}
\end{center}
\end{minipage}
%\hfill
\hspace{-2cm}
\begin{minipage}[t]{0.55\linewidth}
\begin{center}
\includegraphics[width=0.4\textwidth]{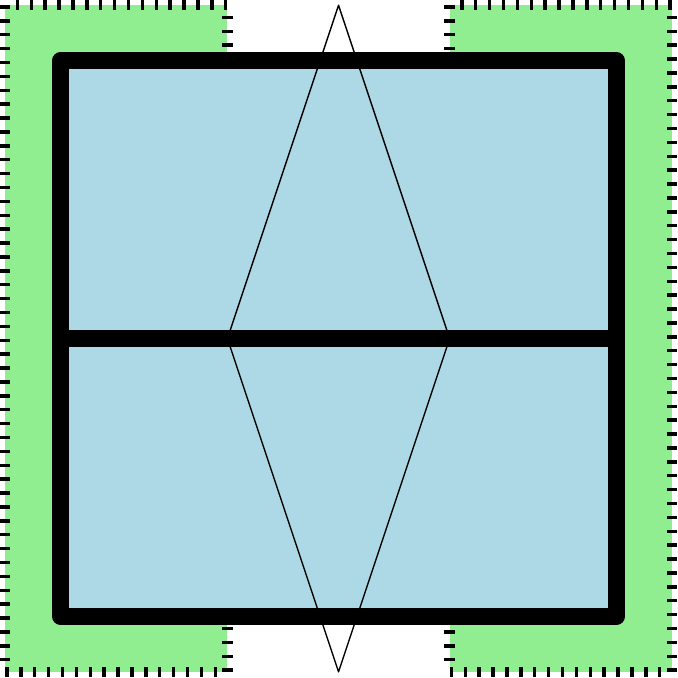}
\par\caption{First family of unconditional sets: The two vertical rectangles, 
the 
two horizontal rectangles and the diamond in the middle; second family: 
The two vertical rectangles, the upper horizontal rectangle taken twice (once 
as itself, and once as a translate of the lower horizontal rectangle) and the 
diamond in the middle.}\label{fig:uncondCounterex}
\end{center}
\end{minipage}
\end{figure}

We prove Theorem~\ref{thm:unifKPmain} in Section~\ref{sec:unifKP}, 
Theorem~\ref{thm:unifKPunionmain} in Section~\ref{sec:unifKPforunions}, and 
finally, Theorem~\ref{thm:strong} in Section~\ref{sec:strong}. 

%%END INTRO 

\section{Proof of Theorem~\ref{thm:unifKPmain}}\label{sec:unifKP}

Theorem~\ref{thm:unifKPmain} clearly follows from the following

\begin{Theorem}\label{thm:unifKP}
 Let $d, N\in\Ze^+, \kind$, and
let $\vect{q}\in\Edn$ be a uniform contraction of $\vect{p}\in\Edn$ with some 
separating value $\lambda\in(0,2]$. If
\begin{enumerate}[(a)]
 \item\label{thmitem:pack}
$N\geq 
\left(1+\frac{2}{\lambda}\right)^d
$,\\ or
 \item\label{thmitem:difficult}
$\lambda\leq\sqrt{2}$ and 
$
N\geq
\left(1+\sqrt{\frac{2d}{d+1}}\right)^d$,
\end{enumerate}
then \eqref{eq:unifKPmain} holds.
\end{Theorem}

In this section, we prove Theorem~\ref{thm:unifKP}.
We may consider a point configuration $\vect{p}\in\Edn$ as a subset of $\Ed$, 
and thus, we may use the notation $\B[\vect{p}]=\bigcap_{i\in N}\B[p_i]$.
We define two quantities that arise naturally. For $d, N\in\Ze^+, \kind$ and 
$\lambda\in(0,2]$, let
\begin{equation*}
 f_k(d,N,\lambda):=\min\left\{ \ivol{\B[\vect{q}]}\st  
 \vect{q}\in\Edn, |q_i-q_j|\leq\lambda \mbox{ for all } i,j\in[N], i\neq j
 \right\},
\end{equation*}
and
\begin{equation*}
 g_k(d,N,\lambda):=\max\left\{ \ivol{\B[\vect{p}]}\st  
 \vect{p}\in\Edn, |p_i-p_j|\geq\lambda \mbox{ for all } i,j\in[N], i\neq j
 \right\}.
\end{equation*}
In this paper, for simplicity, the maximum of the empty set is zero.

Clearly, to establish Theorem~\ref{thm:unifKP}, it will be sufficient to show 
that $f_k\geq g_k$ with the parameters satisfying the assumption of the theorem.

\subsection{Some easy estimates}

We call the following estimate Jung's bound on $f_k$.
\begin{Lemma}\label{lem:jung}
 Let $d, N\in\Ze^+, \kind$ and $\lambda\in(0,\sqrt{2}]$. Then
\begin{equation}\label{eq:jung}
 f_k(d,N,\lambda)\geq \left( 1-\sqrt{\frac{2d}{d+1}}\frac{\lambda}{2} 
\right)^k\ivolbo.
\end{equation}
\end{Lemma}
\begin{proof}[Proof of Lemma~\ref{lem:jung}]
Let $\vect{q}\in\Edn$ be a point configuration in the definition of $f_k$. Then 
Jung's theorem 
\cites{Ju01, DGK63} implies that the circumradius of the set $\{q_i\}$ in 
$\Ed$ is at most 
$\sqrt{\frac{2d}{d+1}}\frac{\lambda}{2}$. It follows that $\B[\vect{q}]$ 
contains a ball of radius $1-\sqrt{\frac{2d}{d+1}}\frac{\lambda}{2}$.
By the monotonicity (with respect to containment) and the degree-$k$ 
homogeneity of $V_k$, the proof of the Lemma is complete.
\end{proof}

The following is a (trivial) packing bound on $g_k$.
\begin{Lemma}\label{lem:pack}
 Let $d, N\in\Ze^+, \kind$ and $\lambda>0$.
\begin{equation}\label{eq:pack}
\mbox{If }
N\left(\frac{\lambda}{2}\right)^d\geq
\left(1+\frac{\lambda}{2}\right)^d, \mbox{ then }
g_k(d,N,\lambda)=0.
\end{equation}
\end{Lemma}
\begin{proof}[Proof of Lemma~\ref{lem:pack}]
Let $\vect{p}\in\Edn$ be such that $|p_i-p_j|\geq\lambda$ for all $i,j\in[N], 
i\neq j$. The balls of radius $\lambda/2$ centered at the points $\{p_i\}$ form 
a packing. By the assumption, taking volume yields that the circumradius of the 
set $\{p_i\}$ is at least one. Hence, $\B[\vect{p}]$ is a singleton or empty.
\end{proof}

We note that we could have a somewhat better estimate in Lemma~\ref{lem:pack} 
if we had a good upper bound on the maximum density of a packing of balls of 
radius $\frac{\lambda}{2}$ in a ball of radius $1+\frac{\lambda}{2}$.

\subsection{Intersections of balls --- An additive Blaschke--Santalo type 
inequality}

Let $X$ be a non-empty subset of $\Ed$ with $\cir(X)\leq\rho$. For $\rho>0$, 
the \emph{$\rho$-spindle convex hull} of $X$ is defined as
\begin{equation*}
 \conv_{\rho}(X):=\B[\B[X,\rho],\rho].
\end{equation*}
It is not hard to see that
\begin{equation}\label{eq:spindlenochange}
 \B[X,\rho]=\B[\conv_{\rho}(X),\rho].
\end{equation}
We say that $X$ is \emph{$\rho$-spindle convex}, if $X=\conv_{\rho}(X)$.

Fodor, Kurusa and V\'igh \cite{FKV16}*{Theorem~1.1}. 
proved a Blaschke--Santalo-type inequality for the volume of spindle convex 
sets.
The main result of this section is an additive version of this inequality, 
which covers all intrinsic volumes.
\begin{Theorem}\label{thm:additiveBS}
Let $Y\subset\Ed$ be a $\rho$-spindle convex set with $\rho>0$, and $, \kind$. 
Then
\begin{equation}\label{eq:additiveBS}
\ivol{Y}^{1/k}+\ivol{\B[Y,\rho]}^{1/k}\leq\rho\ivolbo^{1/k}. 
\end{equation}
\end{Theorem}

Motivated by \cite{FKV16} we observe that
Theorem~\ref{thm:additiveBS} clearly follows from the following proposition 
combined with the Brunn--Minkowski theorem for intrinsic volumes, cf. 
\cite{Ga02}*{equation (74)}.

\begin{Proposition}\label{prop:additiveBScontainment}
Let $Y\subset\Ed$ be a $\rho$-spindle convex set with $\rho>0$. Then
  \begin{equation*}
  Y-\B[Y,\rho]=\B[o,\rho].
  \end{equation*}
\end{Proposition}

Proposition~\ref{prop:additiveBScontainment} has been known (cf. 
\cite{FKV16}*{equation (7)}),
%and \cite{FKV16}*{p. 341} gives \cite{BoFe87}*{Section~63} as a reference, 
but only with some hint on its proof. For the sake of completeness, we present 
the relevant proof here with all the necessary references. We note that instead 
of Proposition~\ref{prop:additiveBScontainment}, one could use a result of 
Capoyleas \cite{Ca96}, according to which, for any $\rho$-spindle convex set 
$Y$, we have that $Y+\B[Y,\rho]$ is a set of constant width $2\rho$, and then 
combine it with the fact that the ball of radius $\rho$ is of the largest 
$k-$th intrinsic volume among sets of constant width $2\rho$ (cf. 
\cite{Sch}*{Section~7.4}).

\begin{proof}[Proof of Proposition~\ref{prop:additiveBScontainment}]
Lemma~3.1 of \cite{BLNP07} implies in a straighforward way that $Y$ slides
freely inside $\B[o,\rho]$. Thus, Theorem 3.2.2 of \cite{Sch}*{Section~3.2} 
yields that
$Y$ is a summand of $\B[o,\rho]$ and so, using Lemma 3.1.8 of 
\cite{Sch}*{Section~3.1}
we get right away that
$$
Y+(\B[o,\rho]\sim Y)=\B[o,\rho],
$$
where  $\sim$ refers to the Minkowski difference with $\B[o,\rho] \sim Y:=  
\cap_{y\in Y}(\B[o,\rho]-y)$.
Thus, we are left to observe that $ \cap_{y\in Y}(\B[o,\rho]-y)=-\B[Y,\rho]$.
\end{proof}

%We may assume that $\rho=1$. 
%For a convex set $K$, let $h_K$ denote its \emph{support function}, that is, 
%$h_K(u):=\sup\{\iprod{u}{x}\st x\in K\}$, for any $u\in\Sedm$.
%We fix an arbitrary unit vector $u\in\Sedm$. It is sufficient to show that
%\begin{equation}\label{eq:supporteq}
% h_{Y}(u)+h_{\B[Y]}(-u)=1.
%\end{equation}

%We leave it as an exercise to show that the left hand side in 
%\eqref{eq:supporteq} is at most 1, and we will show that it is at least 1.

%Let $a\in Y$ be a touching point of $Y$ with its supporting hyperplane 
%of outer normal vector $u$, and let $b:=a-u$. It is sufficient to show that 
%$b\in\B[Y]$. 

%Notice that $Y$ is a 1-spindle convex set. Thus, by 
%\cite{BLNP07}*{Lemma~3.1} we have that $\B[b]\supseteq Y$. It follows that 
%$b\in\B[Y]$, and the proof of Proposition~\ref{prop:additiveBScontainment} 
%is complete.
%\end{proof}

We will need the following fact later, the proof is an exercise for the reader.
 \begin{equation}\label{eq:Bcontainment}
  \B[\vect{q}]\subseteq
  \B\left[
  \bigcup_{i=1}^N \B[q_i,\mu], 1+\mu
  \right],
 \end{equation}
for any $\vect{q}\in\Edn$ and $\mu>0$.

\subsection{A non-trivial bound on \texorpdfstring{$\mathbf{g}$}{g}}

The key in the proof of Theorem~\ref{thm:unifKP} is the following lemma.

\begin{Lemma}\label{lem:gboundnontriv}
 Let $d, N\in\Ze^+,\kind$ and $\lambda\in(0,\sqrt{2}]$. Then
\begin{equation}\label{eq:gboundnontriv}
 g_k(d,N,\lambda)\leq 
 \max\left\{0,
   \left( 1-
     \left( N^{1/d}-1 \right)\frac{\lambda}{2} 
   \right)^k\ivolbo
 \right\}.
\end{equation}
\end{Lemma}

\begin{proof}[Proof of Lemma~\ref{lem:gboundnontriv}]
Let $\vect{p}\in\Edn$ be such that $|p_i-p_j|\geq\lambda$ for all $i,j\in[N], 
i\neq j$. We will assume that $\cir(\vect{p})\leq 1$, otherwise, 
$\B[\vect{p}]=\emptyset$, and there is nothing to prove.

To denote the union of non-overlapping (that is, interior-disjoint) convex 
sets, we use the $\bigsqcup$ operator.

Using \eqref{eq:Bcontainment} with $\mu=\lambda/2$, we obtain
\[
\;\;\;\;\;\;\;\;\;\;\;\;\;\;\;\;\;\;\;\;\;
\ivol{\B[\vect{p}]}
\leq
  \ivol{
   \B\left[
     \bigsqcup_{i=1}^N \B\left[p_i,\frac{\lambda}{2}\right], 1+\frac{\lambda}{2}
     \right]
  }
  \;\;\; =\;\;\;\mbox{ (using } \cir(\vect{p})\leq 1, \mbox{and 
\eqref{eq:spindlenochange})}
\]
\begin{eqnarray*}
  \ivol{
   \B\left[
     \conv_{1+\lambda/2}\left(\bigsqcup_{i=1}^N 
\B\left[p_i,\frac{\lambda}{2}\right]\right), 
       1+\frac{\lambda}{2}
     \right]
  }\leq&& \mbox{ (by \eqref{eq:additiveBS})}\\
  \left[
  \left(1+\frac{\lambda}{2}\right)\ivolbo^{1/k}-
     \ivol{
     \conv_{1+\lambda/2}\left(\bigsqcup_{i=1}^N 
\B\left[p_i,\frac{\lambda}{2}\right]\right)
     }^{1/k}     
  \right]^k\leq&&\\
  \left[
  \left(1+\frac{\lambda}{2}\right)\ivolbo^{1/k}-
    \frac{\lambda}{2}N^{1/d}\ivolbo^{1/k}
  \right]^k,&&
\end{eqnarray*}
where, in the last step, we used the following. 
We have 
\[\ivol[d]{\conv_{1+\lambda/2}\left(\bigsqcup_{i=1}^N 
\B\left[p_i,\frac{\lambda}{2}\right]\right)}\geq
\ivol[d]{(N^{1/d}\lambda/2)\B[o]}.
\]
Thus, by a general form of the isoperimetric inequality (cf. 
\cite{Sch}*{Section~7.4.}) stating that among all convex bodies of given 
(positive) volume
precisely the balls have the smallest $k-$th intrinsic volume for $k=1,\dots, 
d-1$, we have
\[\ivol{\conv_{1+\lambda/2}\left(\bigsqcup_{i=1}^N 
\B\left[p_i,\frac{\lambda}{2}\right]\right)}\geq
\ivol{(N^{1/d}\lambda/2)\B[o]}.
\]
Finally, \eqref{eq:gboundnontriv} follows.
\end{proof}

\subsection{Proof of Theorem~\ref{thm:unifKP}}

(\ref{thmitem:pack}) follows from Lemma~\ref{lem:pack}.
To prove (\ref{thmitem:difficult}), we assume that $\lambda\leq\sqrt{2}$.

By \eqref{eq:jung}, we have
\begin{equation}\label{eq:proofunifKP1}
 \left(\frac{f_k(d,N,\lambda)}{\ivolbo}\right)^{1/k}\geq
 1-\sqrt{\frac{2d}{d+1}}\frac{\lambda}{2}.
\end{equation}
On the other hand, \eqref{eq:gboundnontriv} yields that either 
$g_k(d,N,\lambda)=0$, or
\begin{equation}\label{eq:proofunifKP2}
 \left(\frac{g_k(d,N,\lambda)}{\ivolbo}\right)^{1/k}\leq
 1-\left( N^{1/d}-1 \right)\frac{\lambda}{2}.
\end{equation}
Comparing \eqref{eq:proofunifKP1} and \eqref{eq:proofunifKP2} completes the 
proof of (\ref{thmitem:difficult}), and thus, the proof of 
Theorem~\ref{thm:unifKP}.

\section{Proof of Theorem~\ref{thm:unifKPunionmain}}\label{sec:unifKPforunions}

Theorem~\ref{thm:unifKPunionmain} clearly follows from the next theorem.
For the statement we shall need the following notation and formula. 
Take a regular $d$-dimensional simplex of edge length $2$ in $\Ed$ and then 
draw a $d$-dimensional unit ball around each vertex of the simplex. Let $\sigma 
_d$ denote the ratio of the volume of the portion of the simplex covered by 
balls to the volume of the simplex. It is well known that 
$\sigma_d=\left(\frac{1+o(1)}{e}\right)d2^{-\frac{d}{2}}$, cf. \cite{Ro64}.

\begin{Theorem}\label{thm:unifKPforunions}
	Let $d, N\in\Ze^+$, and let $\vect{q}\in\Edn$ be a uniform contraction of 
	$\vect{p}\in\Edn$ with some  
	separating value $\lambda\in(0, 2)$.

	\item{(a)} If $\lambda\in[\sqrt{2}, 2)$ and  
	$N\geq \left(1+\frac{\lambda}{2}\right)^d\frac{d+2}{2}$, then
	\eqref{eq:unifKPunionmain} holds.

	\item{(b)} If $\lambda\in[0, \sqrt{2})$ and 
	
$N\geq\left(1+\frac{2}{\lambda}\right)^d\sigma_d=\left(\frac{1}{\sqrt{2}}+\frac{
\sqrt{2}}{
	\lambda}\right)^d\left(\frac{1+o(1)}{e}\right)d$, then
	\eqref{eq:unifKPunionmain} holds.

	\item{(c)} If $\lambda\in[0, 1/d^3)$ and 
	$N\geq (2d^2+1)^{d}$, 
	then \eqref{eq:unifKPunionmain} holds.
\end{Theorem}

In this section, we prove Theorem~\ref{thm:unifKPforunions}.

The diameter of $\bigcup_{i=1}^{N}\B[q_i]$ is at most $2+\lambda$. 
Thus, the isodiametric inequality (cf. \cite{Sch}*{Section~7.2.}) 
implies that
\begin{equation}\label{eq:isodiametric}
\ivol[d]{\bigcup_{i=1}^{N}\B[q_i]}\le 
\left(1+\frac{\lambda}{2}\right)^d{\kappa_d}.
\end{equation} 
On the other hand, $\left\{\B[p_i,\lambda/2]\st i=1,\ldots, N\right\}$ is a 
packing of balls.

\subsection{To prove part (a)\texorpdfstring{\nopunct}{}}
in Theorem~\ref{thm:unifKPforunions}, we note that Theorem~2 
of \cite{BL15} implies in a straightforward way that
\[
\frac{N\left(\frac{\lambda}{2}\right)^d{\kappa_d}}{\ivol[d]{\bigcup_{i=1}^{N}\B[
p_i
]}}\leq\frac{d+2}{2}\left(\frac{\lambda}{2}\right)^d.
\]
holds for all $\lambda\in\left[\sqrt{2},2\right)$. Thus, we have
\begin{equation}\label{eq:BL-inequality}
\frac{2N{\kappa_d}}{d+2}\leq \ivol[d]{\bigcup_{i=1}^{N}\B[p_i]}.
\end{equation}
As $N\geq \left(1+\frac{\lambda}{2}\right)^d\frac{d+2}{2}$, the 
inequalities \eqref{eq:isodiametric} and \eqref{eq:BL-inequality} finish the 
proof of part (a).

\subsection{For the proof of part (b),\texorpdfstring{\nopunct}{}} we use a 
theorem of Rogers, discussed in the introduction of \cite{BL15}, according to 
which 
\[
\frac{N\left(\frac{\lambda}{2}\right)^d{\kappa_d}}{\ivol[d]{\bigcup_{i=1}^{N}\B[
p_i
]}}\leq\sigma_d.
\]
holds for all $\lambda\in\left[0,\sqrt{2}\right)$. Thus, we have
\begin{equation}\label{R-inequality}
\frac{N\left(\frac{\lambda}{2}\right)^d{\kappa_d}}{\sigma_d}\leq 
\ivol[d]{\bigcup_{i=1}^{N}\B[p_i]}.
\end{equation}
As 
$N\geq\left(1+\frac{2}{\lambda}\right)^d\sigma_d=\left(\frac{1}{\sqrt{2}}+\frac{
\sqrt{2}}{ 
\lambda}\right)^d\left(\frac{1+o(1)}{e}\right)d$, the inequalities 
\eqref{eq:isodiametric} and 
\eqref{R-inequality} finish the proof of part (b).

\subsection{We turn to the proof of part (c).} 
Note that $\frac{N^{1/d}-1}{2}\geq d^2$. Thus,

\begin{equation*}
\ivol[d]{\bigcup_{i=1}^{N}\B[p_i,\lambda/2]}=
N\left(\frac{\lambda}{2}\right)^d{\kappa_d}\geq
\ivol[d]{\B[o,(d^2+1/2)\lambda]}.
\end{equation*} 
Thus, by the isodiametric inequality, there are two points $p_j$ and $p_k$, 
with $1\leq j< k\leq N$, such that $|p_j-p_k|\geq 2d^2\lambda$. 
Set $h:=|p_j-p_k|/2\geq d^2\lambda$.
Now, $\B[p_j]\cap\B[p_k]$ is symmetric about the perpendicular bisector 
hyperplane $H$ of $p_jp_k$, and $D:=\B[p_j]\cap H= \B[p_k]\cap H$ is a 
$(d-1)$-dimensional ball of radius $\sqrt{1-h^2}$. Let $H^+$ denote 
the half-space bounded by $H$ containing $p_k$. Consider the sector (i.e., 
solid 
cap) $S:=\B[p_j]\cap H^+$, and the cone $T:=\conv(\{p_j\}\cup D)$.
We have two cases.

\emph{Case 1}, when $h\leq\frac{1}{\sqrt{d}}$. Then clearly,
\begin{equation*}
\ivol[d]{\bigcup_{i=1}^{N}\B[p_i]}\geq
\ivol[d]{\B[p_j]\cup\B[p_k]}=
\end{equation*} 
\begin{equation*}
2\kappa_d-2(\ivol[d]{T}+\ivol[d]{S})+2\ivol[d]{T}\geq
\kappa_d+2\ivol[d]{T}=
\kappa_d+2\frac{h}{d}(1-h^2)^{(d-1)/2}\kappa_{d-1}.
\end{equation*} 

The latter expression as a function of $h$ is increasing on the interval 
$[d^2\lambda,1/\sqrt{d}]$. Thus, it is at least
\begin{equation*}
\kappa_d+2\frac{d^2\lambda}{d}(1-d^4\lambda^2)^{(d-1)/2}\kappa_{d-1}
\geq
\kappa_d\left[
1+2d\lambda e^{-d^5\lambda^2}
\right].
\end{equation*} 
By \eqref{eq:isodiametric}, if
\begin{equation}\label{eq:uniongoal}
1+2d\lambda e^{-d^5\lambda^2}
\geq
\left(1+\frac{\lambda}{2}\right)^d
\end{equation} 
holds, then \eqref{eq:unifKPunionmain} follows. 
Using $\lambda\leq d^{-3}$, we obtain \eqref{eq:uniongoal}, 
and thus, Case 1 follows.

\emph{Case 2}, when $h>\frac{1}{\sqrt{d}}$. Then
\begin{equation*}
\ivol[d]{\bigcup_{i=1}^{N}\B[p_i]}\geq
\ivol[d]{\B[p_j]\cup\B[p_k]}\geq
2\kappa_d-2(\ivol[d]{T}+\ivol[d]{S}).
\end{equation*}
Using a well known estimate on the volume of a spherical cap (see e.g. 
\cite{BW03}), we obtain that the latter expression is at least
\begin{equation*}
2\kappa_d\left[
1-\frac{(1-h^2)^{(d-1)/2}}{\sqrt{2\pi(d-1)}h}
\right]\geq
2\kappa_d\left[
1-\frac{(1-1/d)^{(d-1)/2}}{\sqrt{\pi}}
\right]\geq 1.1\kappa_d.
\end{equation*} 
As in Case 1, we compare this with \eqref{eq:isodiametric}, and obtain Case 2. 
This completes the proof of Theorem~\ref{thm:unifKPforunions}.

\section{Proof of Theorem~\ref{thm:strong}}\label{sec:strong}

We prove only \eqref{eq:strong1}, as \eqref{eq:strong2} can be 
obtained in the same way.

Let us start with the point configuration $\vect{p}=(p_1, \ldots, p_N)$ in 
$\Ed$ having coordinates 
\[
p_1=(p_1^{(1)}, \ldots , p_1^{(d)}), 
p_2=(p_2^{(1)}, \ldots , p_2^{(d)}), \ldots , 
p_N=(p_N^{(1)}, \ldots , p_N^{(d)}). 
\]

It is enough to consider the case when $\vect{q}=(q_1,\ldots,q_N)$ is 
such that for each $1\leq i\leq N$ and each $2\leq j\leq d$, we have
\[
q_i^{(j)}=p_i^{(j)}.
\]
In other words, we may assume that all the coordinates of $q_i$, except for the 
first coordinate, are equal to the corresponding coordinate of $p_i$. Indeed, 
if we prove \eqref{eq:strong1} in this case, then, by repeating it for the 
other $d-1$ coordiantes, one completes the proof.

Let $\ell$ be an arbitrary line parallel to the first coordinate axis. Consider 
the sets
\[
\ell_p:=\ell\cap\left(\bigcup_{i=1}^{N}(p_i+K_i)\right) \mbox{ and }
\ell_q:=\ell\cap\left(\bigcup_{i=1}^{N}(q_i+K_i)\right).
\]
Both sets are the union of $N$ (not necessarily disjoint) intervals on $\ell$, 
where the corresponding intervals are of the same length. Moreover, since each 
$K_i$ is unconditional, the sequence of centers of these intervals in $\ell_q$ 
is a contraction of the sequence of centers of these intervals in $\ell_p$. 
Now, \eqref{eq:strong1} is easy to show in dimension 1 (see also \cite{KW}), 
and thus, for the total length (1-dimensional measure) of $\ell_p$ and 
$\ell_q$, we have

\begin{equation}\label{eq:strong-1d}
 \len\left(\ell_p\right)\ge
 \len\left(\ell_q\right).
\end{equation}

Let $H:=\{x=(0,x^{(2)},x^{(3)},\ldots,x^{(d)})\in\Ed\}$ denote the coordinate 
hyperplane orthogonal to the first axis, and for $x\in H$, let $\ell(x)$ denote 
the line parallel to the first coordiante axis that intersects $H$ at $x$.

\begin{equation*}
 \ivol[d]{\bigcup_{i=1}^{N}(p_i+K_i)}=
 \int_H \len\left( \ell(x)\cap\left(\bigcup_{i=1}^{N}(p_i+K_i)\right) \right) 
\di x
 \stackrel{\mbox{by \eqref{eq:strong-1d}}}{\ge}
\end{equation*}
\begin{equation*}
 \int_H \len\left( \ell(x)\cap\left(\bigcup_{i=1}^{N}(q_i+K_i)\right) \right) 
\di x=
 \ivol[d]{\bigcup_{i=1}^{N}(q_i+K_i)},
\end{equation*}
completing the proof of Theorem~\ref{thm:strong}.

\section*{Acknowledgements}
We thank Peter Pivovarov and Ferenc Fodor for our discussions.

K{\'a}roly Bezdek was partially supported by a Natural Sciences and 
Engineering Research Council of Canada Discovery Grant.

M{\'a}rton Nasz{\'o}di was partially supported by the
National Research, Development and Innovation Office (NKFIH) grants: 
NKFI-K119670 and NKFI-PD104744 and by the J\'anos Bolyai Research Scholarship 
of 
the Hungarian Academy of Sciences. Part of his research was carried out during 
a 
stay at EPFL, Lausanne at J{\'a}nos Pach's Chair of Discrete and Computational 
Geometry supported by the Swiss National Science Foundation Grants 
200020-162884 
and 200021-165977.

\bibliographystyle{amsalpha}
\bibliography{biblio}
\end{document}